\documentclass[a4paper,11pt]{article}

\input{styleart.sty}

\title{The free group does not have the finite cover property}
\date{\today}
\author{Rizos Sklinos\thanks{This work has been conducted while the author was supported by a 
Golda Meir postdoctoral fellowship at the Hebrew University of Jerusalem.}}

\begin{document}

\maketitle

\begin{abstract}
We prove that the first order theory of nonabelian free groups  
eliminates the $\exists^{\infty}$-quantifier (in $eq$). Equivalently, since the theory of non abelian free groups is stable, 
it does not have the finite cover property.
\end{abstract}

\section{Introduction}
In \cite{SelaStability} Sela proved that the theory of any torsion-free hyperbolic group is stable. This astonishing result 
renewed the interest of model theorists in the first order theories of such groups. 
Other classes of groups for which the theory of any of their elements is known to be stable,  
is the class of abelian groups and the class of algebraic groups over algebraically closed fields. It is quite interesting 
that model theory has the power to study such diverse classes of groups under a common perspective, the perspective of 
stability.

In this paper we strengthen the above mentioned result. Our main theorem is:

\begin{thmIntro}
Let $\F$ be a nonabelian free group. Then $\mathcal{T}h(\F)^{eq}$ eliminates the $\exists^{\infty}$-quantifier.
\end{thmIntro}

An immediate corollary, since the first order theory of non abelian free groups is stable, is the following:

\begin{corIntro}\label{MainTheorem}
The theory of nonabelian free groups does not have the finite cover property (nfcp).
\end{corIntro}

As a matter of fact the two statements are equivalent assuming stability. 
For a discussion of these notions we refer the reader to section \ref{ModelTheory}.

The paper is structured as follows. Sections \ref{ModelTheory} and \ref{Geometry} serve as introductions to 
model theory and geometric group theory respectively. The purpose is to quickly introduce readers lacking some of 
these backgrounds to the notions and tools that we will use. Needless to say the treatment is by no means complete. 
The central notions in section \ref{ModelTheory} will be the model theoretic imaginaries and the finite cover property. 
In section \ref{Geometry} we will start by explaining Bass-Serre theory, which 
will be also useful for fixing the vocabulary for later use. Then we will pass to  
actions of groups on real trees. These notions will dominate the core of our proofs.

In section \ref{LimitGroups} we define limit groups using the Bestvina-Paulin method (as they were originally defined in \cite{Sel1}),  
a tool that will prove useful throughout the paper. We will continue with the presentation of solid limit groups and we will finish this 
section with the theorem of Sela bounding the number of {\em strictly solid families} of morphisms associated to a solid limit group. 
The previously mentioned theorem, which is Theorem \ref{BoundSolid} in our paper, lies behind our main idea.

In section \ref{TSFT} we present the notion of a {\em tower} and the notion of a {\em test sequence}. A tower 
can be thought of as a group constructed geometrically in a certain fashion, and a test sequence is a sequence of morphisms, 
defined on a tower, satisfying some combinatorial axioms.

Section \ref{CoreMaterial} is the core of our paper. Here is where all the technical results are proved. 
In brief, the idea is that we are able to say something about an element that lives in a tower provided 
that we know the behavior of its images under a test sequence for the tower.

In section \ref{EnvelopeGraded} we present the notion of the {\em Diophantine envelope}. A strong tool whose existence 
has been proved by Sela in \cite{SelaImaginaries}. Diophantine envelopes together with Theorem \ref{BoundSolid} 
are the main pylons of our proof. 

In section \ref{FCP} we bring everything together in order to prove the main results of this paper.

\section{Some model theory}\label{ModelTheory}
In this section we give the model theoretic background needed for the rest of the paper and we place our results into context. 
We will concentrate on the finite cover property and the notion of imaginaries. Since our result might be interesting for a wider audience, 
our exposition is meant to benefit the reader that has no prior knowledge of model theory.

\subsection{Imaginaries}
We fix a first order structure $\mathcal{M}$ and we are interested in the collection of definable sets in $\mathcal{M}$, i.e. all subsets 
of some Cartesian power of $M$ (the domain of $\mathcal{M}$) which are the solution sets of first order formulas (in $\mathcal{M}$). 
In some cases one can easily describe this collection usually thanks to some quantifier elimination result. 
For example, as algebraically closed fields admit (full) quantifier elimination (in the language of rings) 
all definable sets are boolean combinations of algebraic sets, i.e. sets defined by polynomial equations. 
On the other hand, although free groups admit quantifier elimination
down to boolean combinations of $\forall\exists$ formulas (see \cite{Sel5,Sel5bis}), the ``basic'' definable sets are not so easy to describe. 
 
Suppose $X$ is a definable set in $\mathcal{M}$. One might ask whether there is a canonical way to 
define $X$, i.e. is there a tuple $\bar{b}$ and a formula $\psi(\bar{x},\bar{y})$ such that $\psi(\mathcal{M},\bar{b})=X$ 
but for any other $\bar{b}'\neq \bar{b}$, $\psi(\mathcal{M},\bar{b}')\neq X$? 

To give a positive answer to the above mentioned question one has to move to a mild expansion of 
$\mathcal{M}$ called $\mathcal{M}^{eq}$. Very briefly $\mathcal{M}^{eq}$ is constructed from $\mathcal{M}$ 
by adding a new sort for each $\emptyset$-definable equivalence relation, $E(\bar{x},\bar{y})$, together 
with a class function $f_E:M^n\rightarrow S_E(M)$, where $S_E(M)$ (the domain of the new sort corresponding to $E$) 
is the set of all $E$-equivalence classes. The elements in these new sorts are called {\em imaginaries}. 
In $\mathcal{M}^{eq}$, it is not hard to see that one can assign to each definable set a canonical parameter in the sense discussed above. 
Let us also record the following useful lemma.

\begin{lemma}\label{ImagSorts}
Let $\phi(x_1,\ldots,x_l)$ be a first order formula in the expanded language $\mathcal{L}_{\mathcal{M}}^{eq}$, i.e. the 
language $\mathcal{L}_{\mathcal{M}}$ expanded with function symbols and variables corresponding to the imaginary sorts. 
For each $i\leq l$, suppose $x_i$ belongs to the sort corresponding to some $\emptyset$-definable equivalence relation $E_i$. 

Then there exists a first order $\mathcal{L}_{\mathcal{M}}$-formula $\psi(\bar{z}_1,\ldots,\bar{z}_l)$ such that 
$$\mathcal{M}^{eq}\models \forall\bar{z}_1,\ldots,\bar{z}_l(\phi(f_{E_1}(\bar{z}_1),\ldots,f_{E_l}(\bar{z}_l))\leftrightarrow 
\psi(\bar{z}_1,\ldots,\bar{z}_l))$$
\end{lemma}

\subsection{Stability}
For a thorough introduction to stability we refer the reader to \cite{PillayStability}.  
We will now quickly record the notions and results we are interested in. 

\begin{definition} 
Let $\phi(\bar{x},\bar{y})$ be a first order $\mathcal{L}$-formula. Then $\phi(\bar{x},\bar{y})$ has the order property in an 
$\mathcal{L}$-structure $\mathcal{M}$, if there exist infinite sequences $(\bar{a}_n)_{n<\omega}$, $(\bar{b}_n)_{n<\omega}$ such that 
$\mathcal{M}\models\phi(\bar{a}_i,\bar{b}_j)$ if and only if $i<j$. 
\end{definition}

\begin{definition}
A first order theory $T$ is stable if no formula has the order property in any model of $T$.
\end{definition}

Stable theories can also be introduced as theories that support a good independence relation, for a discussion 
towards this end we refer the reader to \cite[Section 2]{PeSkFor}.

The following property has been introduced in \cite{KeislerFCP} by Keisler.

\begin{definition}
Let $\phi(\bar{x},\bar{y})$ be a first order $\mathcal{L}$-formula. Then $\phi(\bar{x},\bar{y})$ has the finite cover property in an 
$\mathcal{L}$-structure $\mathcal{M}$, if for arbitrarily large $n$ there exist $\bar{b}_1,\ldots,\bar{b}_n$ such that: 
$$\mathcal{M}\models \lnot\exists\bar{x}\bigwedge_{i\leq n} \phi(\bar{x},\bar{b}_i)$$ 
But for each $j\leq n$:
$$\mathcal{M}\models \exists\bar{x}\bigwedge_{i\leq n, i\neq j} \phi(\bar{x},\bar{b}_i)$$
\end{definition}

As before the property passes to a first order theory as follows:

\begin{definition}
A first order theory $T$ does not have the finite cover property (or $T$ has nfcp), if no formula 
has the finite cover property in any model of $T$.
\end{definition}

The following fact is almost immediate:

\begin{fact}
If a first order theory is unstable then it has the finite cover property.
\end{fact}

There is an equivalent point of view towards nfcp if we assume stability. Towards this end we define:

\begin{definition}
A first order theory $T$ eliminates the ``there exists infinitely many'' quantifier (or $T$ eliminates the $\exists^{\infty}$-quantifier) if 
to each first order formula $\phi(\bar{x},\bar{y})$ we can assign a natural number $n_{\phi}$ such that for 
any model $\mathcal{M}$ of $T$, for any $\bar{b}\in\mathcal{M}$, if $\abs{\phi(\mathcal{M},\bar{b})}<\infty$, then $\abs{\phi(\mathcal{M},\bar{b})}< n_{\phi}$.
\end{definition}

\begin{fact}\label{nfcp}
A stable first order theory $T$ does not have the finite cover property if and only if $T^{eq}$ eliminates 
the ``there exists infinitely many'' quantifier. 
\end{fact}

\subsection{Some model theory of the free group}

In this subsection we will record some fundamental theorems concerning the model theory of nonabelian free groups. 

We start by defining some basic families of imaginaries. 

\begin{definition}\label{Imaginaries}
Let $\F$ be a nonabelian free group. The following equivalence relations in $\F$ are called basic.
\begin{itemize}
 \item[(i)] $E_1(a,b)$ if and only if there is $g\in \F$ such that $a^g=b$. (conjugation)
 \item[$(ii)_m$] $E_{2_m}((a_1,b_1),(a_2,b_2))$ if and only if either $b_1=b_2=1$ or $b_1\neq 1$ and $C_{\F}(b_1)=C_{\F}(b_2)=\langle b \rangle$ and
$a_1^{-1}a_2\in\langle b^m \rangle$. ($m$-left-coset)
 \item[$(iii)_m$] $E_{3_m}((a_1,b_1),(a_2,b_2))$ if and only if either $b_1=b_2=1$ or $b_1\neq 1$ and $C_{\F}(b_1)=C_{\F}(b_2)=\langle b \rangle$ and
$a_1a_2^{-1}\in\langle b^m \rangle$. ($m$-right-coset)
 \item[$(iv)_{m,n}$] $E_{4_{m,n}}((a_1,b_1,c_1),(a_2,b_2,c_2))$ if and only if either 
$a_1=a_2=1$ or $c_1=c_2=1$ or $a_1,c_1\neq 1$ and $C_{\F}(a_1)=C_{\F}(a_2)=\langle a \rangle$ and $C_{\F}(c_1)=C_{\F}(c_2)=\langle c \rangle$
  and there is $\gamma\in \langle a^m \rangle$ and $\epsilon\in \langle c^n \rangle$ such that $\gamma b_1 \epsilon=b_2$. ($m,n$-double-coset)
\end{itemize}

\end{definition}

It is almost immediate that $m$-left cosets eliminate $m$-right cosets (and vice versa), 
so from now on we are austere and forget about the $m$-right-cosets (this observation, which does not use any special property of nonabelian free groups, 
has been pointed out to us by Ehud Hrushovski).

Sela proved the following theorem concerning imaginaries in nonabelian free groups (see \cite[Theorem 4.4]{SelaImaginaries}).

\begin{theorem}\label{Elim}
Let $\F$ be a nonabelian free group. Let $E(\bar{x},\bar{y})$ be a definable equivalence relation in $\F$, with $\abs{\bar{x}}=m$.
Then there exist $k,l<\omega$ and a definable relation:
$$R_E \subseteq \F^m \times \F^k \times S_1(\F) \times \ldots \times S_l(\F)$$
such that:
\begin{itemize}
 \item[(i)] each $S_i(\F)$ is one of the basic sorts;
 \item[(ii)] for each $\bar{a}\in \F^m$ , $\abs{R_E(\bar{a},\bar{z})}$ is uniformly bounded (i.e. the bound does not depend on $\bar{a}$);
 \item[(iii)] $\F\models\forall\bar{z}(R_E(\bar{a},\bar{z})\leftrightarrow R_E(\bar{b},\bar{z}))$ if and only if $E(\bar{a},\bar{b})$.
\end{itemize}
\end{theorem}

Nonabelian free groups share the same first-order theory (see \cite{Sel6}, \cite{KharlampovichMyasnikov}), which we call the theory of 
the free group.

\begin{fact}[Sela, Kharlampovich-Miasnikov]
Let $\F_n$ be the free group of rank $n$. Then $\F_n$ is elementarily equivalent to $\F_m$ for any $m,n\geq 2$. 
\end{fact}

The following astonishing result has been proved by Sela.

\begin{fact}[Sela \cite{SelaStability}]
The theory of the free group is stable.
\end{fact}

\section{Some geometry}\label{Geometry}

In this section we record some background material from geometric group theory. Our choice is biased towards our needs in this paper. 
In particular, we will start by giving some quick introduction 
to Bass-Serre theory. We will then continue by defining $\R$-trees and certain group actions on them that will play an essential role in the  proofs of our results. 

\subsection{Bass-Serre theory}\label{BS}

Bass-Serre theory gives a structure theorem for groups acting on (simplicial) trees.
It describes a group (that acts on a tree) 
as a series of {\em amalgamated free products} and {\em HNN extensions}. The mathematical notion that 
contains these instructions is called a graph of groups. For a complete treatment we refer the reader to 
\cite{SerreTrees}.

We start with the definition of a graph.

\begin{definition}
A graph $G(V,E)$ is a collection of data that consists of two sets $V$ (the set of vertices) and $E$ (the set of edges) 
together with three maps:
\begin{itemize}
 \item an involution $\bar{\phantom{1}}:E\to E$, where $\bar{e}$ is called the inverse of $e$;
 \item $\alpha:E\to V$, where $\alpha(e)$ is called the initial vertex of $e$; and
 \item $\tau:E\to V$, where $\tau(e)$ is called the terminal vertex of $e$. 
\end{itemize}
so that $\bar{e}\neq e$, and $\alpha(e)=\tau(\bar{e})$ for every $e\in E$.
\end{definition}

An {\em orientation} of a graph $G(V,E)$ is a choice of one edge in the couple $(e,\bar{e})$ for every $e\in E$. We denote 
an oriented graph by $G^{+}(V,E)$.

For our purposes simplicial trees can also be viewed as combinatorial objects: a {\em tree} is a connected graph without a circuit. 

\begin{definition}[Graph of Groups]
A graph of groups $\mathcal{G}:=(G(V,E), \{G_u\}_{u\in V}, \{G_e\}_{e\in E},$ $\{f_e\}_{e\in E})$ consists of the following data:
\begin{itemize}
 \item a graph $G(V,E)$;
  \item a family of groups $\{G_u\}_{u\in V}$, i.e. a group is attached to each vertex of the graph;
  \item a family of groups $\{G_e\}_{e\in E}$, i.e. a group is attached to each edge of the graph. Moreover, $G_{e}=G_{\bar{e}}$;
  \item a collection of injective morphisms $\{f_e:G_e\to G_{\tau(e)} \ | \ e\in E\}$, i.e. each edge group comes 
  equipped with two embeddings to the incident vertex groups. 
\end{itemize}

\end{definition}

The fundamental group of a graph of groups is defined as follows.

\begin{definition}
Let $\mathcal{G}:=(G(V,E), \{G_u\}_{u\in V}, \{G_e\}_{e\in E}, \{f_e\}_{e\in E})$ be a graph of groups. Let $T$ be a maximal subtree of $G(V,E)$. 
Then the fundamental group, $\pi_1(\mathcal{G},T)$, of $\mathcal{G}$ with respect to $T$ is the group given by the following presentation:
$$\langle \{G_u\}_{u\in V}, \{t_e\}_{e\in E} \ | \ t^{-1}_e=t_{\bar{e}} \ \textrm{for} \ e\in E, t_e=1 \ \textrm{for} \ e\in T, 
f_e(a)=t_ef_{\bar{e}}(a)t_{\bar{e}} \ \textrm{for} \ e\in E \ a\in G_e\rangle$$
\end{definition}

\begin{remark}
It is not hard to see that the fundamental group of a graph of groups does not depend on the 
choice of the maximal subtree up to isomorphism (see \cite[Proposition 20, p.44]{SerreTrees}).
\end{remark}

In order to give the main theorem of Bass-Serre theory we need the following definition.

\begin{definition}  
Let $G$ be a group acting on a simplicial tree $T$ without inversions,
denote by $\Lambda$ the corresponding quotient graph and by $p$ the quotient map $T \to \Lambda$. 
A Bass-Serre presentation for the action of $G$ on $T$ is a triple $(T^1, T^0, \{\gamma_e\}_{e\in E(T^1)\setminus E(T^0)})$ consisting of
\begin{itemize}
\item a subtree $T^1$ of $T$ which contains exactly one edge of $p^{-1}(e)$ for each edge $e$ of $\Lambda$;
\item a subtree $T^0$ of $T^1$ which is mapped injectively by $p$ onto a maximal subtree of $\Lambda$;
\item a collection of elements of $G$, $\{\gamma_e\}_{e\in E(T^1)\setminus E(T^0)}$, such that if $e=(u,v)$ with $v\in T^1\setminus T^0$, 
then $\gamma_e\cdot v$ belongs to $T^0$.
\end{itemize}
\end{definition}

\begin{theorem}
Suppose $G$ acts on a simplicial tree $T$ without inversions. Let $(T^1,T^0, \{\gamma_e\})$ be a Bass-Serre presentation for the action. 
Let $\mathcal{G}:=(G(V,E), \{G_u\}_{u\in V}, \{G_e\}_{e\in E}, \{f_e\}_{e\in E})$ be the following graph of groups:
\begin{itemize}
 \item $G(V,E)$ is the quotient graph given by $p:T\to \Gamma$;
 \item if $u$ is a vertex in $T^0$, then $G_{p(u)}=Stab_G(u)$;
 \item if $e$ is an edge in $T^1$, then $G_{p(e)}=Stab_G(e)$;
 \item if $e$ is an edge in $T^1$, then $f_{p(e)}:G_{p(e)}\to G_{\tau(p(e))}$ is given by the identity 
 if $e\in T^0$ and by conjugation by $\gamma_e$ if not.
\end{itemize}

Then $G$ is isomorphic to $\pi_1(\mathcal{G})$. 
\end{theorem}

Among splittings of groups we will distinguish those 
with some special type vertex groups called {\em surface type vertex groups}.

\begin{definition}
Let $G$ be a group acting on a tree $T$ without inversions and $(T_1,T_0,\{\gamma_e\})$ be a Bass-Serre presentation for this action. 
Then a vertex $v\in T^0$ is called a surface type vertex if the following conditions hold:
\begin{itemize}
 \item $\Stab_G(v)=\pi_1(\Sigma)$ for a connected compact surface $\Sigma$ with non-empty boundary;
 \item For every edge $e\in T_1$ adjacent to $v$, $\Stab_G(e)$ embeds onto a maximal boundary 
 subgroup of $\pi_1(\Sigma)$, and this induces a one-to-one correspondence between the 
 set of edges (in $T^1$) adjacent to $v$ and the set of boundary components of $\Sigma$.
\end{itemize}

\end{definition}

\subsection{Real trees}

Real trees (or $\R$-trees) generalize simplicial trees in the following way.  

\begin{definition}
A real tree is a geodesic metric space in which 
for any two points there is a unique arc that connects them.
\end{definition}

When we say that a group $G$ acts on an real tree $T$ we will always mean an action by isometries.

Moreover, an action $G\curvearrowright T$ of a group $G$ on a real tree $T$ is called {\em non-trivial} if there is no 
globally fixed point and {\em minimal} if there is no proper $G$-invariant subtree. Lastly, an action is called {\em free} 
if for any $x\in T$ and any non trivial $g\in G$ we have that $g\cdot x\neq x$.

We continue by recalling some families of group actions on real trees. 

\begin{definition}
Let $G\curvearrowright^{\lambda} T$ be a minimal action of a finitely generated group $G$ on a real tree $T$. Then we say:
\begin{itemize}
 \item[(i)] $\lambda$ is of simplicial type, if every orbit $G.x$ is discrete in $T$. In this case $T$ is simplicial 
 and the action can be analyzed using Bass-Serre theory;
 \item[(ii)] $\lambda$ is of axial type, if $T$ is isometric to the real line $\R$ and $G$ acts with dense orbits, 
 i.e. $\overline{G.x}=T$ for every $x\in T$;
 \item[(iii)] $\lambda$ is of surface type, if $G=\pi_1(\Sigma)$ where $\Sigma$ is a surface with (possibly empty) boundary 
 carrying an arational measured foliation and $T$ is dual to $\tilde{\Sigma}$, i.e. $T$ is the lifted leaf space in $\tilde{\Sigma}$ 
 after identifying leaves of distance $0$ (with respect to the pseudo-metric induced by the measure);
\end{itemize}
\end{definition}

\begin{fact}
Suppose $\pi_1(\Sigma)$ acts on a real tree $T$ by a surface type action. Then the action is  ``almost free'', i.e. 
only elements that belong to subgroups that correspond to the boundary components fix points in $T$ and segment stabilizers are trivial. 
In particular when $\Sigma$ has empty boundary the action is free (see \cite{MorganShalenFree}).
\end{fact}

We will use the notion of a graph of actions in order to glue real trees equivariantly.  
We follow the exposition in \cite[Section 1.3]{GuirardelRTrees}.

\begin{definition}[Graph of actions]
A graph of actions $(G\curvearrowright T,\{Y_u\}_{u\in V(T)},\{p_e\}_{e\in E(T)})$ consists of the following data:
\begin{itemize}
 \item A simplicial type action $G\curvearrowright T$; 
 \item for each vertex $u$ in $T$ a real tree $Y_u$;
 \item for each edge $e$ in $T$, an attaching point $p_e$ in $Y_{\tau(e)}$.
\end{itemize}

Moreover:
\begin{enumerate}
 \item $G$ acts on $R:=\{\coprod Y_u : u\in V(T)\}$ so that $q:R\to V(T)$ with $q(Y_u)=u$ is $G$-equivariant;
 \item for every $g\in G$ and $e\in E(T)$, $p_{g\cdot e}=g\cdot p_e$.
\end{enumerate}

\end{definition}

To a graph of actions $\mathcal{A}:=(G\curvearrowright T,\{Y_u\}_{u\in V(T)},\{p_e\}_{e\in E(T)})$ we can assign an $\R$-tree $Y_{\mathcal{A}}$ endowed with 
a $G$-action. Roughly speaking this tree will be $\coprod_{u\in V(T)} Y_{u}/\sim$, where the equivalence relation $\sim$ identifies $p_e$ with 
$p_{\bar{e}}$ for every $e\in E(T^{+})$. We say that a real $G$-tree $Y$ {\em decomposes as a graph of actions} $\mathcal{A}$, if there is 
an equivariant isometry between $Y$ and $Y_{\mathcal{A}}$. 

Assume a real $G$-tree $Y$ decomposes as a graph of actions. Then a useful property is that $Y$ is covered by $(Y_u)_{u\in V(T)}$ 
and moreover these trees intersect ``transversally''.

\begin{definition}
Let $Y$ be an $\R$-tree and $(Y_i)_{i\in I}$ be a family of subtrees that cover $Y$. Then we call this covering 
a transverse covering if the following conditions hold:
\begin{itemize}
 \item for every $i\in I$, $Y_i$ is a closed subtree;
 \item for every $i,j\in I$ with $i\neq j$, $Y_i\cap Y_j$ is either empty or a point;
 \item every segment in $Y$ is covered by finitely many $Y_i$'s.
\end{itemize}

\end{definition}

The next lemma is by no means hard to prove (see \cite[Lemma 4.7]{GuirardelRnTrees}).

\begin{lemma}\label{TraCov}
Let $\mathcal{A}:=(G\curvearrowright T,\{Y_u\}_{u\in V(T)},\{p_e\}_{e\in E(T)})$ be a graph of actions. 
Suppose $G\curvearrowright Y$ decomposes as the graph of actions $\mathcal{A}$.  
Then $(Y_u)_{u\in V(T)}$ is a transverse covering of $Y$. 
\end{lemma}

\section{Limit groups}\label{LimitGroups}

In this section we define the basic objects of this paper, namely {\em graded limit groups}. 

We will start by defining {\em limit groups} in a geometric way: a limit group is a quotient of a finitely generated group by 
the kernel of an action (obtained in a specific way) of the group on a real tree. 

In the second subsection we consider the class of graded limit groups, which can be thought of as a class of group couples, $(G,H)$, 
where $G$ is a limit group and $H$ a finitely generated subgroup of $G$. 
We record some theorems about {\em solid limit groups} 
(i.e. graded limit groups that satisfy some extra conditions) that underpin our main idea for proving the main theorem of this paper.

\subsection{The Bestvina-Paulin method}
The following construction is credited to Bestvina \cite{BesLimit} and Paulin \cite{PaulinGromov} independently. 

We fix a finitely generated group $G$ and we consider the set of non-trivial 
equivariant pseudometrics $d:G\times G\to \R^{\geq 0}$, denoted by $\mathcal{ED}(G)$. 
We equip $\mathcal{ED}(G)$ with the compact-open topology (where $G$ is given the discrete topology). Note 
that convergence in this topology is given by: 
$$(d_i)_{i<\omega}\to d\ \ \textrm{if and only if} \ \ d_i(1,g)\to d(1,g)\ \ (\textrm{in $\R$}) \ \ \textrm{for any}\ \ g\in G$$

Is not hard to see that $\R^+$ acts cocompactly on $\mathcal{ED}(G)$ by rescaling, thus 
the space of {\em projectivised equivariant pseudometrics} on $G$ is compact.

We also note that any based $G$-space $(X,*)$ (i.e. a metric space with a distinguished point equipped with an action of 
$G$ by isometries) gives rise to an equivariant pseudometric on $G$ as follows: $d(g,h)=d_X(g\cdot *,h\cdot *)$. 

We say that a sequence of $G$-spaces $(X_i,*_i)_{i<\omega}$ converges to a $G$-space $(X,*)$, if the 
corresponding pseudometrics induced by $(X_i,*_i)$ converge to the pseudometric induced by $(X,*)$ in $\mathcal{PED}(G)$. 

A morphism $h:G\to H$ where $H$ is a finitely generated group induces an action of $G$ on  
$\mathcal{X}_H$ (the Cayley graph of $H$) in the obvious way, thus making $\mathcal{X}_H$ a $G$-space. We have:

\begin{lemma}\label{LimitAction} 
Let $\F$ be a nonabelian free group. Let $(h_n)_{n<\omega}:G\to\F$ be a sequence of pairwise non-conjugate morphisms. 
Then for each $n<\omega$ there exists a base point $*_n$ 
in $\mathcal{X}_{\F}$ such that the sequence of $G$-spaces $(\mathcal{X}_{\F},*_n)_{n<\omega}$ has a convergent subsequence to 
a real $G$-tree $(T,*)$, where the action of $G$ on $T$ is nontrivial.
\end{lemma}

Every point in the limiting tree $T$, obtained as above, can be approximated by 
a sequence of points of the converging subsequence in the following sense.

\begin{lemma}\label{ApproximatingSequences}
Assume $(\mathcal{X}_{\F}, *_n)_{n<\omega}$ converges to $(T,*)$ as in Lemma \ref{LimitAction}. Then for any $x\in T$, the following hold:
\begin{itemize}
 \item there exists a sequence $(x_n)_{n<\omega}$ such that $\hat{d}_n(x_n,g\cdot *_n)\to d_T(x,g\cdot *)$ for any $g\in G$, where $\hat{d}_n$ 
 denotes the rescaled metric of $\mathcal{X}_{\F}$, we call such a sequence an approximating sequence;
 \item if $(x_n)_{n<\omega},(x'_n)_{n<\omega}$ are two approximating sequences for $x\in T$, then $\hat{d}_n(x_n,x'_n)\to 0$;
 \item if $(x_n)_{n<\omega}$ is an approximating sequence for $x$, then $(g\cdot x_n)_{n<\omega}$ is an approximating sequence for 
 $g\cdot x$;
 \item if $(x_n)_{n<\omega},(y_n)_{n<\omega}$ are approximating sequences for $x,y$ respectively, then $\hat{d}_n(x_n,y_n)\to d_T(x,y)$.
\end{itemize}

\end{lemma}

\begin{definition}[Sela] 
A group $L$ is a limit group if it can be obtained as the quotient of a finitely generated group 
$G$ by $Ker\lambda$ where $\lambda$ is an action of $G$ on a real tree obtained as in Lemma \ref{LimitAction}. 
\end{definition}

Among other properties of limit groups Sela proved (see \cite[Lemma 1.4(ii)]{Sel7}).

\begin{proposition}
Let $L$ be a freely indecomposable limit group. Then $L$ admits a cyclic $JSJ$ decomposition.  
\end{proposition}

A cyclic $JSJ$ decomposition can be thought of as a ``universal'' cyclic splitting that describes all cyclic splittings of the limit group $L$. 
The vertex groups of this decomposition can be partitioned to {\em rigid} ones, i.e. those that can be conjugated in a vertex group in any cyclic splitting 
of $L$, and {\em nonrigid}. The nonrigid vertex groups are either free abelian or surface type. 

\begin{proposition}
Let $L$ be a limit group. Then $L$ is CSA (i.e. every maximal abelian subgroup of $L$ is malnormal).
\end{proposition}

\subsection{Graded limit groups}

Graded limit groups occur naturally in the understanding of how a solution set of 
a system of equations in $\F$ changes as the parameters of 
the system change (see \cite[Section 10]{Sel1}). We will be interested in a particular subclass of graded limit groups, 
namely {\em solid limit groups}.

A graded limit group is a limit group together with a fixed finitely generated subgroup. Any limit group can be seen as graded limit group after choosing a finitely generated subgroup of it. 
As a matter of fact, one can see limit groups as special cases of graded limit groups where the chosen subgroup is 
the trivial one.

In order to define solid limit groups, we first need to present the class of (graded) short morphisms.

\begin{definition}
Let $G$ be a limit group which is freely indecomposable with respect to a finitely generated subgroup $H$. 
We fix a finite generating set, $\Sigma$ for $G$ and a basis $\bar{a}$ for $\F:=\F(\bar{a})$.  
A morphism $h:G\rightarrow \F$ is short with respect to $Mod_H(G)$ if for every $\sigma\in Mod_H(G)$ and 
every $g\in \F$ that commutes with $h(H)$ we have that $max_{s\in\Sigma}\abs{h(s)}_{\F}\leq max_{s\in\Sigma}\abs{Conj(g)\circ h\circ \sigma}_{\F}$.
\end{definition}

Before the next definition we recall that a sequence of morphisms $h_n:G\rightarrow\F$ is {\em convergent} if 
for every $g\in G$ there exists $n\in \N$ such that for every $m\geq n$, either the image of $g$ under $h_m$ is trivial, or non trivial. 
The {\em stable kernel} of a convergent sequence $\Ker h_n$ is defined as the subgroup of elements of $G$ that are eventually trivial under $h_n$. 
    
\begin{definition}
Let $G$ be a limit group which is freely indecomposable with respect to a finitely generated subgroup $H$. 
Let $(h_n)_{n<\omega}:G\rightarrow \F$ be a convergent sequence of short morphisms with respect to $Mod_H(G)$. 
Then we call $G/\Ker h_n$ a shortening quotient of $G$ with respect to $H$.
\end{definition}

\begin{definition}[Solid limit group]
Let $S$ be a limit group and $H$ be a finitely generated subgroup of $S$. 
Suppose $S$ is freely indecomposable with respect to $H$. Then 
$S$ is solid with respect to $H$ if there exists a shortening quotient  
of $S$ with respect to $H$ which is isomorphic to $S$.
\end{definition}

\begin{definition}
Let $Sld$ be a solid limit group with respect to a finitely generated subgroup $H$ with generating set $\Sigma_H$. 
Let $(h_n)_{n<\omega}:Sld\rightarrow \F$. We call $(h_n)_{n<\omega}$ a flexible sequence if for every $n$, either: 
\begin{itemize}
 \item the morphism $h_n=h'_n\circ\eta_n$, where $\eta_n:Sld\twoheadrightarrow \F*\Gamma$ for some group $\Gamma$, 
 $H$ is mapped onto $\F$ by $\eta_n$, $h'_n:\F*\Gamma\rightarrow\F$ stays the identity on $\F$, and 
 $\eta_n$ is short with respect to $Mod_H(Sld)$; or   
 \item the morphism $h_n$ is short with respect to $Mod_H(Sld)$ and moreover 
 $$max_{g\in B_n}\abs{h_n(g)}_{\F}> 2^n(1+max_{s\in\Sigma_H}\abs{h_n(s)}_{\F})$$ 
 where $B_n$ is the ball of radius $n$ in the Cayley graph of $Sld$.
\end{itemize}
If $(h_n)_{n<\omega}:Sld\rightarrow\F$ is a convergent flexible sequence, then we call $Sld/\Ker h_n$ a flexible quotient of $Sld$.
\end{definition}

Flexible quotients of a solid limit group are proper (see \cite[Lemma 10.4(ii)]{Sel1}). Moreover one can 
define a partial order and an equivalence relation on the class of flexible quotients of a solid limit group. 
Let $Sld$ be a solid limit group with respect to a finitely generated subgroup $H$ 
and $\eta_i:Sld\twoheadrightarrow Q_i$ for $i\leq 2$ be flexible quotients with their canonical quotient maps. 
Then $Q_2\leq Q_1$ if $ker\eta_1\subseteq ker\eta_2$. And $Q_1\sim Q_2$ if there exists $\sigma\in Mod_H(Sld)$ 
such that $ker(\eta_1\circ\sigma)=ker\eta_2$.

\begin{theorem}[Sela]
Let $Sld$ be a solid limit group with respect to a finitely generated subgroup $H$. Assume that $Sld$ admits a flexible quotient. Then there exist  
finitely many classes of maximal flexible quotients. 
\end{theorem}

A morphism from a solid limit group to a free group that factors through one of the maximal flexible quotients 
(after precomposition by a modular automorphism) is called {\em flexible}. A morphism which is not flexible is called {\em solid}.

In \cite[Theorem 2.9]{Sel3} Sela proved:

\begin{theorem}\label{BoundSolid}
Suppose $Sld$ is a solid limit group with respect to a finitely generated subgroup $H$. 
Then there exists a natural number $n$ such that any morphism from $H$ to $\F$ has at most $n$ extensions to 
strictly solid morphisms from $Sld$ to $\F$ that belong to different strictly solid families.
\end{theorem}

In order to define a strictly solid morphism and a strictly solid family we will make use of the relative $JSJ$ decomposition of a solid limit group $Sld$  
with respect to its distinguished subgroup $H$. 

\begin{definition}
Let $Sld$ be a solid limit group with respect to a finitely generated subgroup $H$. Let $JSJ_H(Sld)$ be the relative 
$JSJ$ decomposition of $Sld$ with respect to $H$. 

Let $h:Sld\rightarrow \F$ be a morphism. Then $g:Sld\rightarrow\F$ is $JSJ_H(Sld)$-equivalent to $h$ if it agrees with $h$ in the rigid vertex group that contains $H$ 
and it differs in every other rigid group only by a conjugation.  
\end{definition}

The above defined relation is clearly an equivalence relation on the set of morphisms from a solid limit group to a fixed nonabelian free group.

A morphism $h:Sld\rightarrow\F$ is {\em striclty solid} if it is not $JSJ_H(Sld)$-equivalent to a flexible morphism. The equivalent class of 
a strictly solid morphism is called a {\em strictly solid family}.

\section{Towers and test sequences}\label{TSFT}

{\em Towers} or more precisely groups that have the structure of an $\omega$-residually free tower (in Sela's terminology), were 
introduced in \cite[Definition 6.1]{Sel1} and they provide examples of groups for which Merzlyakov's 
theorem (see \cite{Mer}) naturally generalizes (up to some tuning).  
 
After defining towers we will be interested in {\em test sequences} defined on them. These are sequences of morphisms 
from groups that have the structure of a tower to nonabelian free groups. Test sequences have been used extensively 
through out the work of Sela on Tarski's problem.  
As their formal definition is highly technical we will only abstract, in a series of facts, 
all the essential properties  needed for our purposes.

\subsection{Towers}

We will start this subsection by defining the main building blocks of a tower, 
namely free abelian flats and surface flats.

\begin{definition}[Free abelian flat]
Let $G$ be a group and $H$ be a subgroup of $G$. Then $G$ has the structure of a free abelian flat over $H$,  
if $G$ is the amalgamated free product $H*_A(A\oplus \Z^l)$ where $A$, called the peg, is a maximal abelian subgroup of $H$.
\end{definition}
  
Before moving to the definition of a hyperbolic floor we recall that if $H$ is a subgroup of a group $G$ then a morphism $r:G\to H$ 
is called a {\em retraction} if $r$ is the identity on $H$.

\begin{definition}[Hyperbolic floor]
Let $G$ be a group and $H$ be a subgroup of $G$. Then $G$ has the structure of a hyperbolic floor over $H$,  
if $G$ acts minimally and nontrivially on a tree $T$ so that the action admits a Bass-Serre presentation 
$(T^1, T^0, \{\gamma_e\})$ that satisfies the following conditions:
\begin{itemize}
\item there is exactly one vertex $w$ in $T^0$ which is not of surface type with $Stab_G(w)=H$; 
\item the subtree $T^1$ is bipartite between vertices in the orbit of $w$  and surface type vertices; 
\item either there exists a retraction $r:G\to H$ that sends every surface type vertex group to a non abelian image  
or $H$ is cyclic and there exists a retraction $r': G * \Z \to H * \Z$ with the same property.
\end{itemize}
\end{definition}

\begin{figure}[ht!]
\centering
\includegraphics[width=.7\textwidth]{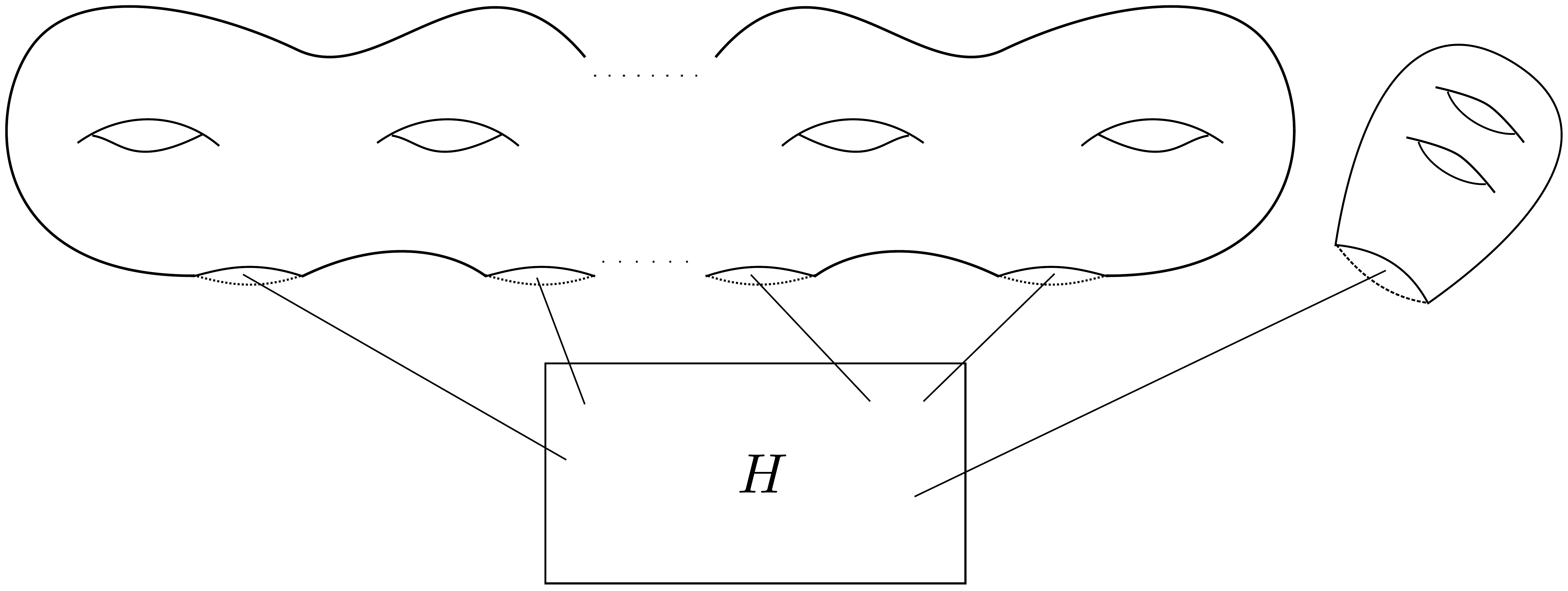}
\caption{A graph of groups corresponding to a hyperbolic floor}\label{Fig1}
\end{figure}

If a group has the structure of a hyperbolic floor (over some subgroup), and the corresponding Bass-Serre presentation 
contains just one surface type vertex then we call the hyperbolic floor {\em a surface flat}. 

The following lemma is immediate.

\begin{lemma}
Suppose $G$ has a hyperbolic floor structure over $H$. Then there exists a finite sequence 
$$G:=G^m>G^{m-1}>\ldots >G^0:=H$$
such that, for each $i<m$, $G^{i+1}$ has a surface flat structure over $G^i$. 
\end{lemma}

We use surface flats and free abelian flats in order to define towers.

\begin{definition}\label{Tower}
A group $G$ has the structure of a tower (of height $m$) over a subgroup $H$ if there 
exists a sequence $G=G^m>G^{m-1}>\ldots>G^0>H$, where $G^0:=H*\F$ with $\F$ a finite rank free group and for each $i$, $0< i<m$, one of the following holds:
\begin{itemize}
 \item[(i)] the group $G^{i+1}$ has the structure of a surface flat over $G^i$ in which $H$ is a subgroup of $G^i$;
 \item[(ii)] the group $G^{i+1}$ has the structure of a free abelian flat over $G^i$ in which $H$ is a subgroup of $G^i$ and the peg cannot be 
conjugated into a group that corresponds to an abelian flat of a previous floor .
\end{itemize}

\end{definition}

\begin{figure}[ht!]
\centering
\includegraphics[width=.7\textwidth]{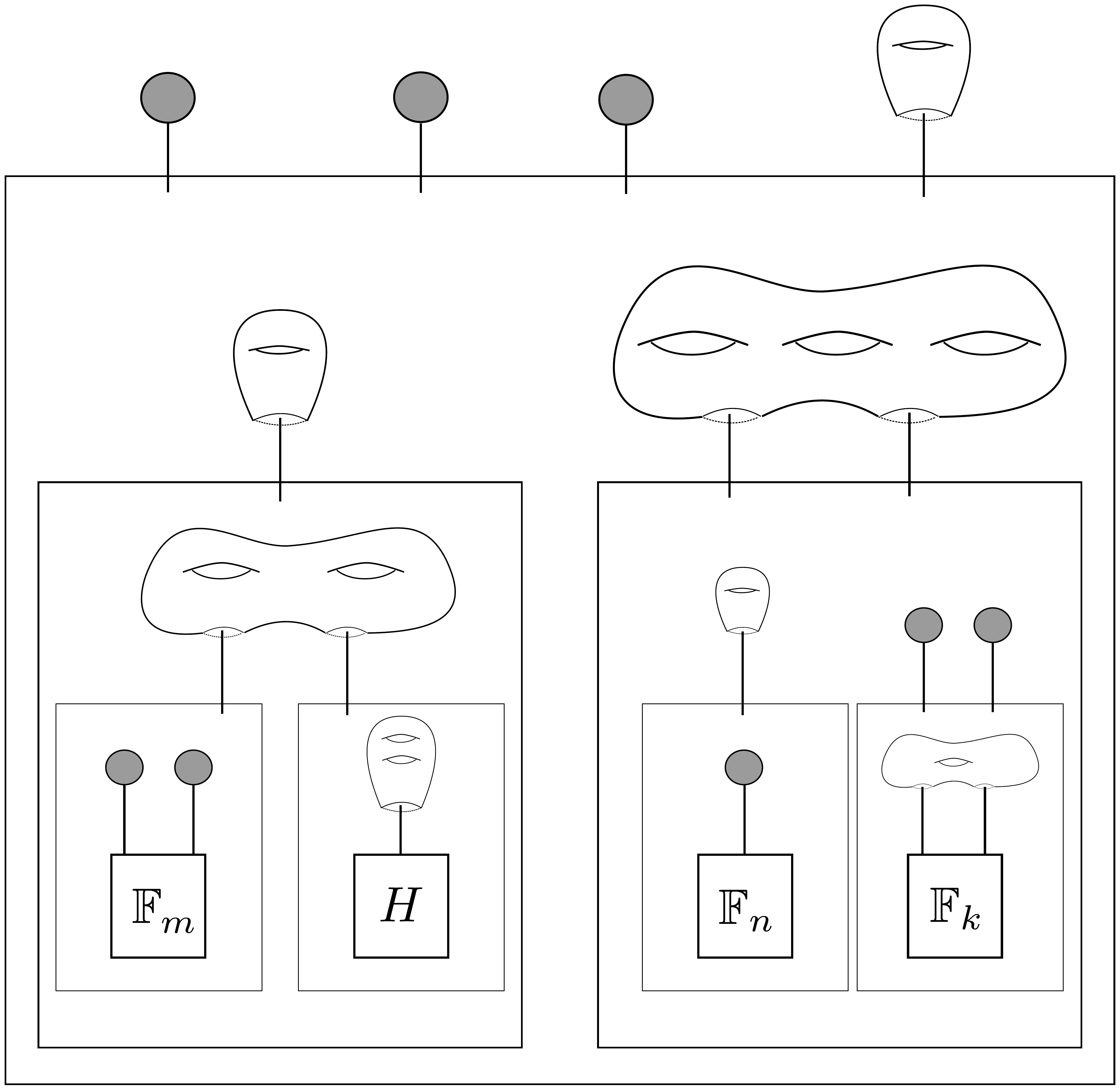}
\caption{A tower over $H$.}\label{FigTower}
\end{figure}

It will be useful to collect the information witnessing that a group admits the structure of a tower. 
Thus, we define:

\begin{definition}
Suppose $G$ has the structure of a tower over $H$ (of height $m$). Then the tower (over $H$) corresponding to $G$, 
denoted by $\mathcal{T}(G,H)$, 
is the following collection of data:
$$(G,\mathcal{A}(G^m,G^{m-1}),\mathcal{A}(G^{m-1},G^{m-2}),\ldots, \mathcal{A}(G^1,G^0), H)$$
where by $\mathcal{A}(G^{i+1},G^i)$ we denote the action of $G^{i+1}$ on a tree (together with the Bass-Serre presentation) that witnesses that 
$G^{i+1}$ has one of the forms of Definition \ref{Tower} over $G^i$.
\end{definition}

The height of an element of a group that has the structure of a tower is defined as follows:

\begin{definition}
Suppose $G$ has the structure of a tower $\mathcal{T}:=\mathcal{T}(G,H)$ and $g$ is an element in $G$.  
Then:
\begin{itemize}
 \item $height_{\mathcal{T}}(g)=0$, if $g\in G^0$;
 \item $height_{\mathcal{T}}(g)=l+1$, if $g\in G^{l+1}\setminus G^l$;
\end{itemize}
\end{definition}

The following lemma is an immediate consequence of Theorem 1.30 in \cite{BFNotesOnSela}

\begin{lemma}
Suppose $G$ has the structure of a tower over a limit group $L$. Then $G$ is a limit group.
\end{lemma}

\subsection{Test sequences}\label{TS}

\begin{convention}
When we refer to Lemma \ref{LimitAction} we will always choose the sequence of base points to be the sequence of trivial elements. 
\end{convention}

If a group $G$ has the structure of a tower (over $\F$) then a test sequence is a sequence of morphisms from $G$ to $\F$ 
that satisfy certain combinatorial conditions. These conditions depend on the structure of the tower and their description is 
highly complicated (see \cite[p. 222]{Sel2}). 
Thus, we prefer to only give the properties that will be of importance to our purpose.

\begin{fact}[Free product limit action]\label{GroundFloorFact}
Let $\mathcal{T}(G,\F)$ be a limit tower and $(h_n)_{n<\omega}:G\to\F$ 
be a test sequence for $\mathcal{T}(G,\F)$. 

Suppose $G^{0}$ is the free product of $\F$ with a finitely generated free group $\F_l$. 
Then, any subsequence of $(h_n\upharpoonright G^{0})_{n<\omega}$ 
that converges, as in Lemma \ref{LimitAction}, induces a faithful action of $G^{0}$ on a 
based real tree $(Y,*)$, with the following properties:
   \begin{figure}[ht!]
   \centering
   \includegraphics[width=.5\textwidth]{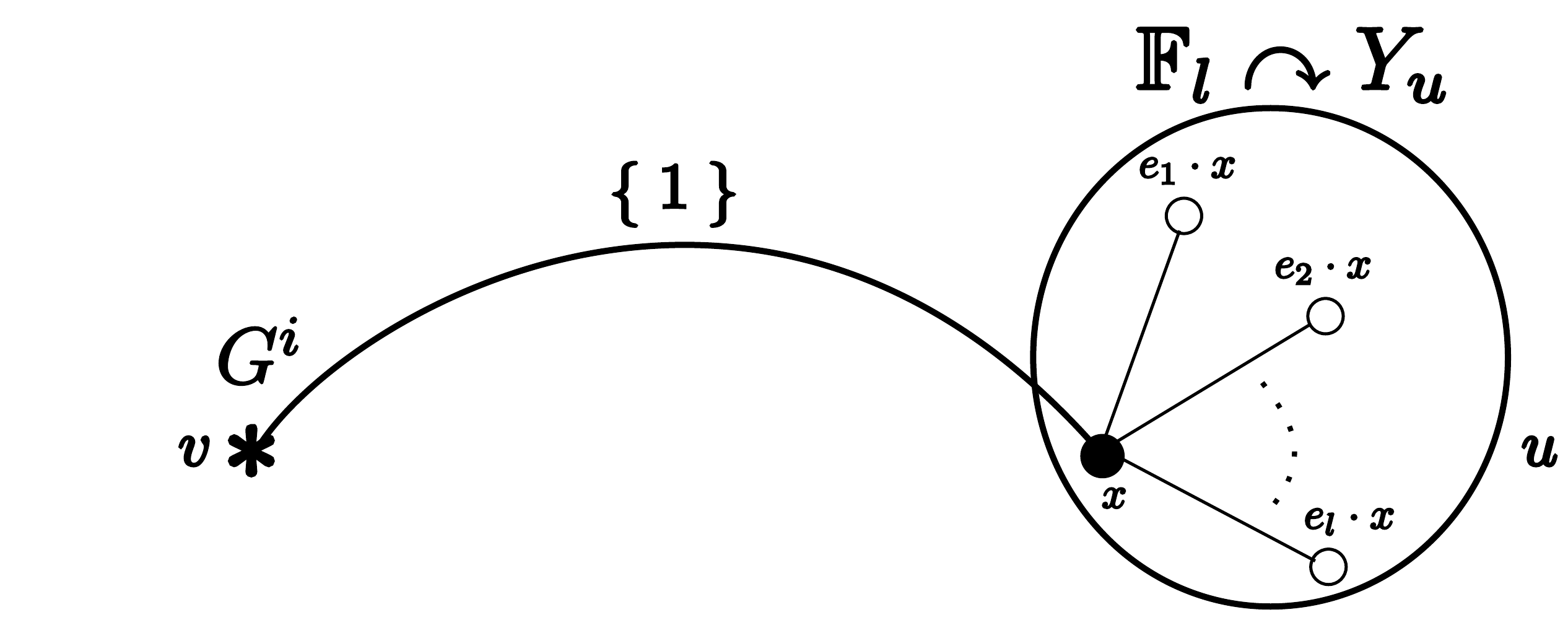}
   \caption{A Bass-Serre presentation in the case of a free group.}\label{Fig4}
   \end{figure}
\begin{itemize}
   \item the action $G^{0}\curvearrowright Y$ decomposes as 
   a graph of actions $(G^{0}\curvearrowright T, \{Y_u\}_{u\in V(T)}$ $,\{ p_e\}_{e\in E(T)})$;
   \item the Bass-Serre presentation for $G^{0}\curvearrowright T$, $(T_1=T_0,T_0)$, is a segment $(u,v)$;
   \item the action $Stab_G(u):=\F_l\curvearrowright Y_u$ is a simplicial type action, its Bass-Serre presentation, 
   $(Y_u^1,Y_u^0,t_1,\ldots,t_l)$ consists of a ``star graph'' $Y_u^1:=\{(x,b_1),\ldots,(x,b_l)\}$ with all of its edges 
   trivially stabilized, a point 
   $Y_u^0=x$ which is trivially stabilized and Bass-Serre elements $t_i=e_i$, for $i\leq l$;
   \item $Y_v$ is a point and $Stab_G(v)$ is $\F$;
   \item the edge $(u,v)$ is trivially stabilized.
 \end{itemize}
\end{fact}

\begin{fact}[Surface flat limit action]\label{SurfaceFact}
Let $\mathcal{T}(G,\F)$ be a limit tower and $(h_n)_{n<\omega}:G\to\F$ 
be a test sequence with respect to $\mathcal{T}(G,\F)$.
  
Suppose $G^{i+1}$ is a surface flat over $G^i$, witnessed by $\mathcal{A}(G^{i+1},G^i)$. 

Then, any subsequence of $(h_n\upharpoonright G^{i+1})_{n<\omega}$ that converges, as in Lemma \ref{LimitAction},  
induces a faithful action of $G^{i+1}$ on a based real tree $(Y,*)$, with the following properties:
   \begin{figure}[ht!]
   \centering
   \includegraphics[width=.7\textwidth]{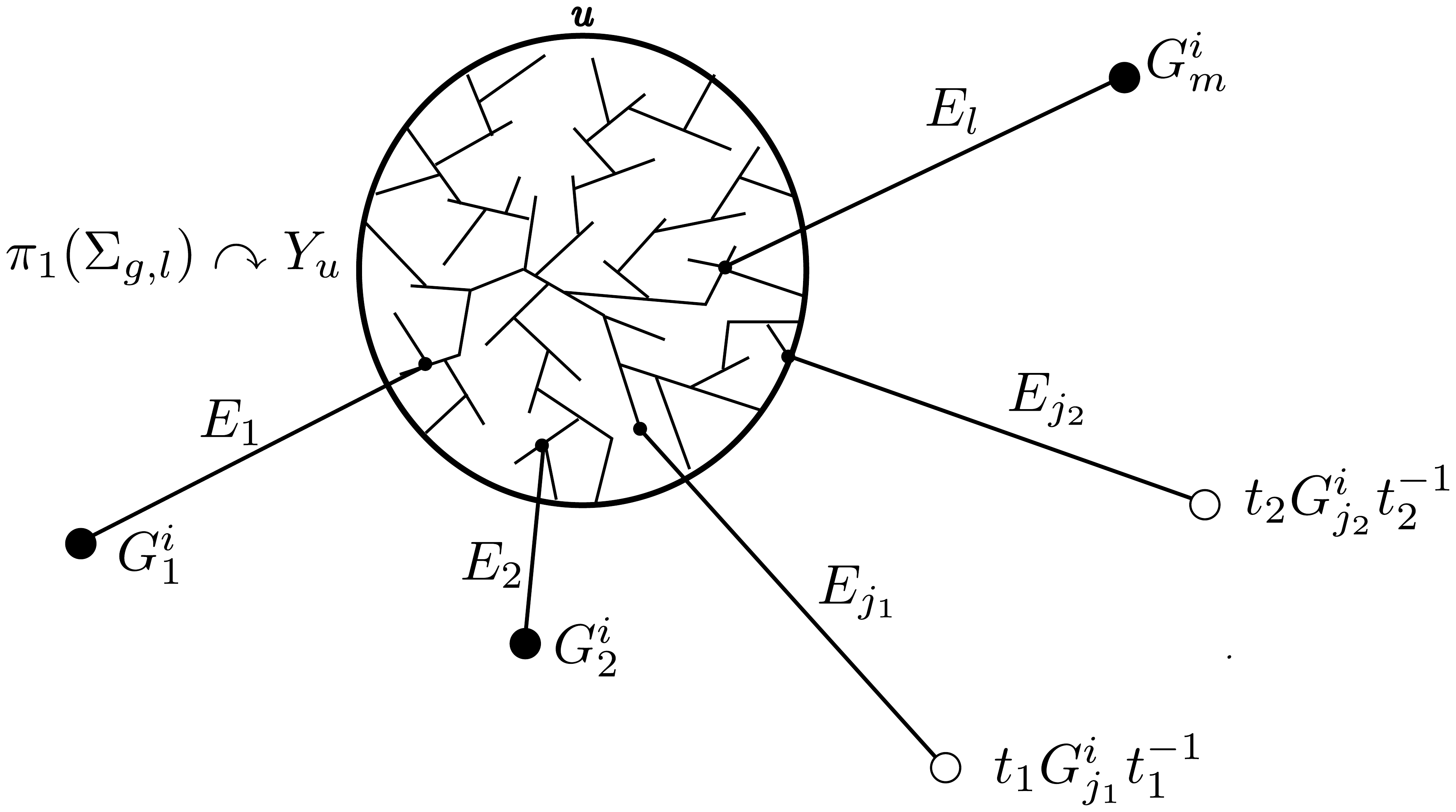}
   \caption{A Bass-Serre presentation for $G^{i+1}\curvearrowright T$, in the case of a surface flat.}\label{Fig5}
   \end{figure}
 \begin{enumerate}
  \item $G^{i+1}\curvearrowright Y$ decomposes as a graph of actions $(G^{i+1}\curvearrowright T, \{Y_u\}_{u\in V(T)}, \{ p_e\}_{e\in E(T)})$, 
  with the action $G^{i+1}\curvearrowright T$ being identical to $\mathcal{A}(G^{i+1},G^i)$; 
  \item the Bass-Serre presentation, $(T^{1},T^{0},\{t_e\})$, for the action of $G^{i+1}$ on $T$, is identical with 
  the Bass-Serre presentation of the surface flat splitting $\mathcal{A}(G^{i+1},G^{i})$;
  \item if $v$ is not a surface type vertex then $Y_v$ is a point stabilized by the corresponding $G^i_j$ for some $j\leq m$;
  \item if $u$ is the surface type vertex, then $Stab_G(u)=\pi_1(\Sigma_{g,l})$ and 
  the action $Stab_G(u)\curvearrowright Y_u$ is a surface type action coming from $\pi_1(\Sigma_{g,l})$;
  \item edge stabilizers and Bass-Serre elements in $(T^{1},T^{0},\{t_e\})$ are as in the Bass-Serre presentation 
  for $\mathcal{A}(G^{i+1},G^{i})$. 
 \end{enumerate}
\end{fact}

\begin{fact}[Abelian flat limit action]\label{AbelianFact}
Let $\mathcal{T}(G,\F)$ be a limit tower and $(h_n)_{n<\omega}:G\to\F$ 
be a test sequence with respect to $\mathcal{T}(G,\F)$. 

Suppose $G^{i+1}=G^i*_A(A\oplus\Z)$ is obtained from $G^i$ by gluing a free abelian flat along $A$ (where $A$ is a maximal 
abelian subgroup of $G^i$). 

Then any subsequence of $(h_n\upharpoonright G^{i+1})_{n<\omega}$ that converges, as in Lemma \ref{LimitAction}, 
induces a faithful action of $G^{i+1}$ on a based real tree $(Y,*)$, with the following properties:
 \begin{enumerate}
  \item the action of $G^{i+1}$ on $Y$, $G^{i+1}\curvearrowright Y$, decomposes as 
  a graph of actions $(G^{i+1}\curvearrowright T, \{Y_u\}_{u\in V(T)},\{ p_e\}_{e\in E(T)})$;
   \item the Bass-Serre presentation for $G^{i+1}\curvearrowright T$, $(T_1=T_0,T_0)$, is a segment $(u,v)$;
   \item $Stab_G(u):=A\oplus\Z\curvearrowright Y_u$ is a simplicial type action, its Bass-Serre presentation, 
   $(Y_u^1,Y_u^0,t_e)$ consists of a segment $Y_u^1:=(a,b)$ whose stabilizer is $A$, a point 
   $Y_u^0=a$ whose stabilizer is $A$ and a Bass-Serre element $t_e$ which is $z$;
   \item $Y_v$ is a point and $Stab_G(v)$ is $G^i$;
   \item the edge $(u,v)$ is stabilized by $A$.
 \end{enumerate}
\end{fact}

\section{An infinitude of images}\label{CoreMaterial}
In this section we prove all the technical results that are in the core of this paper. 
We fix a nontrivial finitely generated free group $\F$.

\begin{theorem}\label{UniRealSort}
Let $G$ be a group that has the structure of a tower $\mathcal{T}(G,\F)$. Suppose $g\in G$ and  
$\{h_n(g) \ | \ n<\omega\}$ is finite for some test sequence $(h_n)_{n<\omega}$ for $\mathcal{T}(G,\F)$. 
Then $g$ belongs to $\F$. 

In particular, the element $g$ takes a unique value under any morphism from $G$ to $\F$ (that stays the identity on $\F$).
\end{theorem}
\begin{proof}
The proof is by induction on the height of $g$. 

For the base case, suppose $height_{\mathcal{T}}(g)=0$ and $G^0:=\F*\F_l$. The sequence $(h_n)_{n<\omega}$ restricted to $G^0$ admits a subsequence that converges and 
induces an action of $G^0$ on a based real tree as in Fact \ref{GroundFloorFact}. Since $g$ takes finitely many values under $(h_n)_{n<\omega}$ it must 
fix the base point. But the stabiliser of the base point is $\F$ itself. Thus $g$ belongs to $\F$. 

For the induction step, assume the claim is true for every element of height less than or equal to $i$, we show it is true for elements of height $i+1$.  
Since $g$ takes finitely many values under $(h_n)_{n<\omega}$, we can use the Bestvina-Paulin method to obtain, from a subsequence of $(h_n)_{n<\omega}$ restricted to $G^{i+1}$, 
an action of $G^{i+1}$ on a based real tree $(Y,*)$ so that $g$ fixes the base point. Recall that, by the definition of a tower, 
$G^{i+1}$ could either be a surface flat or an abelian flat over $G^i$. In either case Facts \ref{SurfaceFact}, \ref{AbelianFact} 
tell us that the base point is stabilised by $G^i$, thus $g$ must belong to $G^i$ and the induction hypothesis concludes the proof. 
\end{proof}

We get the same result for conjugacy classes. For $g$ an element of $\F$ we denote by $[g]$ its  
conjugacy class in $\F$.

\begin{theorem}\label{UniConjSort}
Let $G$ be a group that has the structure of a tower $\mathcal{T}(G,\F)$.  
Suppose $g\in G$ and $\{[h_n(g)] \ | \ n<\omega\}$ is finite, 
for some test sequence $(h_n)_{n<\omega}$ for $\mathcal{T}(G,\F)$. 
Then $g$ can be conjugated in $\F$. 

In particular, 
the conjugacy class of the image of $g$ under any morphism from $G$ to $\F$ (that stays the identity on $\F$) is unique.
\end{theorem}
\begin{proof}
The proof is similar to the previous proof, it is by induction on the height of $g$.

Since $[h_n(g)]$ is finite 
we may assume that $h_n(g)=c^{\gamma_n}$ for some sequence $(\gamma_n)_{n<\omega}\subset\F$ and an element $c\in\F$. 
Moreover, we may assume that $c$ does not commute with $\gamma_n$ for all but finitely many $n$ and that $|\gamma_n|$, 
the length of $\gamma_n$ under a fixed basis of $\F$, tends to infinity as $n$ goes to infinity. If one of these assumptions 
is not met we could find a limiting action where $g$ fixes the base point in the limit real tree, and proceed as in the proof 
of Theorem \ref{UniRealSort}. 

For the base case, suppose the height of $g$ is $0$ and $G^0:=\F*\F_l$. We can use the Bestvina-Paulin method 
to obtain a limiting action on a based real tree $(Y,*)$ satisfying the conclusions of Fact \ref{GroundFloorFact}. Under the above assumptions, 
since the length of $\gamma_n$ is eventually bounded by the rescaling factor, the 
sequence $(\gamma_n)_{n<\omega}$ approximates a point, say $p$, in the limit real tree. 
The element $g$ fixes this point since, by Lemma \ref{ApproximatingSequences}, 
the distance between $p$ and $g\cdot p$ in $Y$ is the limit of the distances $\hat{d}_n(\gamma_n,h_n(g)\cdot\gamma_n)$, which is $0$. 
In order to conclude, we want to prove that $p$ is a translate of the base point.  
We consider the structure of the limit real tree $Y$ decomposed as a graph of actions given by Fact \ref{GroundFloorFact}. By Lemma \ref{TraCov} $Y$ is covered 
by translates of $Y_u$, the component which is acted on by $\F_l$. Thus, $p$ can be translated into $Y_u$ and consequently a conjugate of 
$g$ fixes a point in $Y_u$. But, by the type of the action on $Y_u$, only a translate of the base point inside it can have a nontrivial stabiliser. 
This latter observation concludes the proof. 
 
For the induction step, assume the claim is true for every element of height less than or equal to $i$, we show it is true for elements of height $i+1$. 
The proof is not much different than the base case. Recall that, by the definition of a tower, the group  
$G^{i+1}$ could either be a surface flat or an abelian flat over $G^i$. In either case we can obtain a limiting action on a based real tree $(Y,*)$ 
so that the conclusion of either Fact \ref{SurfaceFact} or Fact \ref{AbelianFact} hold. As in the base case the element $g$ fixes a point $p$ 
that can be translated into the component of the real tree that is acted on by either the surface flat or the abelian flat. In both cases, 
the only points that admit nontrivial stabilisers are translates of the base point, hence $g$ can be conjugated into $G^i$ and 
the induction hypothesis concludes the proof.
\end{proof}

For couples (respectively triples) in the sort for left-$m$-cosets (respectively $(m,n)$-double cosets) we get a slightly different result. We 
split the results in three lemmata. 

\begin{lemma}\label{UniCent}
Let $G$ be a group that has the structure of a tower $\mathcal{T}(G,\F)$. Let 
$(h_n)_{n<\omega}:G\to\F$ be a test sequence for $\mathcal{T}(G,\F)$ and $g\in G$ such that  
$\{C_{\F}(h_n(g)) \ | \ n<\omega\}$ is finite. 

Then either $g$ belongs to $\F$, or there exists a free abelian flat $G^{i+1}:=G^i*_A(A\oplus \Z^m)$ of the tower  
$\mathcal{T}(G,\F)$ whose peg $A$ is a (maximal abelian) subroup of $\F$, and $g$ belongs to $A\oplus \Z^m$. 

In particular, the centraliser of the image of $g$ under any morphism from $G$ to $F$, that stays the identity on $\F$, is unique.
\end{lemma}
\begin{proof}
The proof is by induction on the height of $g$.

Since $\{C_{\F}(h_n(g)) \ | \ n<\omega\}$ is finite, we may assume that $h_n(g)=c^{k_n}$, where $c$ has no root 
and $k_n$ tends to infinity as $n$ goes to infinity. 

For the base case, suppose the height of $g$ is $0$ and $G^0:=\F*\F_l$. We can use the Bestvina-Paulin method 
to obtain a limiting action on a based real tree $(Y,*)$ satisfying the conclusions of Fact \ref{GroundFloorFact}. 
Under the above assumptions, either $g$ fixes the base point and thus $g$ is in $\F$, 
or the segment $[*,g\cdot *]$ is not trivial. In the latter case, the segment $[*,g\cdot *]$ is fixed by 
the centraliser $C_{\F}(c)$. To see this we can use Lemma \ref{ApproximatingSequences}. The base point is fixed by any power of $c$ 
as the test sequence is constant on $\F$, on the other hand the distance $d_Y(g\cdot *, c^m\cdot(g\cdot *))$ is the 
limit of the distances $\hat{d}_n(c^{k_n}, c\cdot c^{k_n})$ that tends to $0$. 
But, according to Fact \ref{GroundFloorFact}, all segments in the real tree $(Y,*)$ 
are trivially stabilised, a contradiction. 

Assume the result is true for any element $g$ of height less or equal to $i$. We show it for elements of height $i+1$. 
Recall that by the definition of a tower the $i+1$-th flat could either be a surface flat or a free abelian flat. In both cases, 
we follow the argument of the base case. We may assume that $g$ does not fix the base point, hence the nontrivial 
segment $[*,g\cdot *]$ is fixed by the centraliser $C_{\F}(c)$.  We now separate the cases. 

Suppose first that the $i+1$-th flat is a surface flat. As in the base case, all nontrivial segments are trivially stabilised, thus we obtain 
a contradiction.

In the case the $i+1$-th flat is a free abelian flat we first show that the segment $[*,g\cdot *]$ must be contained entirely 
in a translate of the component, say $Y_u$, which is acted on by the abelian flat. 
Suppose not, then we will have at least two subsegments $I_1, I_2$ of $[*,g\cdot *]$ 
contained in two distinct translates of $Y_u$. Each component has a stabiliser which is a conjugate of the peg of the free abelian flat, 
as a matter of fact these conjugates must meet trivially. Thus, the whole segment is trivially stabilised, a contradiction. 
In addition, since there exists a nontrivial segment which is stabilised by $C_{\F}(c)$ we may assume that $A=C_{\F}(c)$.  
Now since the segment $[*,g\cdot *]$ belongs to the component stabilised by $A\oplus \Z^m$, there exists 
an element in $A\oplus \Z^m$, say $z$, such that $g\cdot *=z\cdot *$, in particular $g=z\cdot\gamma$ for some element 
$\gamma$ in $G^i$. Since $h_n(g)=h_n(z)h_n(\gamma)$ we see that $h_n{\gamma}=c^{l_n}$ for some sequence of integers $l_n$.  Thus, 
by the induction hypothesis, either $g$ belongs to $\F$ and more precisely to $C_{\F}(c)$, or there exists an abelian flat 
$G^{j+1}=G^j*_B (B\oplus \Z^r)$ where $B=C_{\F}(c)$, a contradiction to the structure of a tower, since the peg $A$ cannot be 
glued to (a conjugate of) a previous peg. Thus, the element $g$ has the required form. 
\end{proof}

Let $k$ be a natural number. We denote by $C_{\F}^k(c)$ the subgroup of $k$-powers of the centraliser 
$C_{\F}(c)$ of a nontrivial element $c$ in $\F$.

\begin{lemma}\label{LeftCosets}
Let $G$ be a group that has the structure of a tower $\mathcal{T}(G,\F)$. Let 
$(h_n)_{n<\omega}:G\to\F$ be a test sequence for $\mathcal{T}(G,\F)$ and $c\in\F\setminus\{1\}$.  

Let $g\in G$ such that $\{h_n(g)\cdot C^k_{\F}(c) \ | \ n<\omega\}$ is finite. Then either $g\in\F$ or there exist an element 
$d\in\F$ and a free abelian flat $G^{i+1}:=G^i*_A(A\oplus \Z^m)$ of the tower  
$\mathcal{T}(G,\F)$ whose peg $A$ is the group $C_{\F}(c)$, and $d^{-1}g$ belongs to  
$A\oplus \Z^m$.

In particular, there are at most $k$ cosets $h(g)\cdot C^k_{\F}(c)$ under any morphism $h$ from $G$ to $F$ 
that stays the identity on $\F$.

\end{lemma}

\begin{proof}
By the hypothesis we can refine the given sequence 
so that $h_n(g)\cdot C_{\F}^k(c)=d\cdot C_{\F}^k(c)$ for some $d\in\F$. Thus, $h_n(g)=d\cdot c^{m_n}$ for 
some sequence $(m_n)_{n<\omega}$ of elements in $k\cdot \Z$. Hence $h_n(d^{-1}\cdot g)= c^{m_n}$ and following the 
proof of Lemma \ref{UniCent} we conclude.
\end{proof}

Lemmata \ref{UniCent} and \ref{LeftCosets} give the following corollary.

\begin{theorem}\label{BoundLeftCosets}
Let $G$ be a group that has the structure of a tower $\mathcal{T}(G,\F)$. Let 
$(h_n)_{n<\omega}:G\to\F$ be a test sequence for $\mathcal{T}(G,\F)$. Let $g_1,g_2$ be elements in $G$, $k$ be a natural number 
and $E_{2_k}$ be the equivalence relation of $k$-left-cosets (see Definition \ref{Imaginaries}).
Suppose $\{[(h_n(g_1),h_n(g_2))]_{E_{2,k}} \ | \ n<\omega\}$ is finite. 

Then the set $\{[(h(g_1),h(g_2))]_{E_{2,k}} \ | \ h:G\rightarrow\F, \ \ h\upharpoonright\F=Id \}$ is bounded by $k$.
\end{theorem}
\begin{proof}
Since the set $\{[(h_n(g_1),h_n(g_2))]_{E_{2,k}} \ | \ n<\omega\}$ is finite we have that the set 
$\{C_{\F}(h_n(g_2)) \ | \ $ $ n<\omega\}$ is finite. Thus, by Lemma \ref{UniCent} there exists 
a non-trivial element $c\in\F$ so that, 
for any morphism from $G$ to $\F$ that stays the identity on $\F$, the centraliser of the image of $g_2$ is $C_{\F}(c)$. 

On the other hand, again by the finiteness of $\{[(h_n(g_1),h_n(g_2))]_{E_{2,k}} \ | \ n<\omega\}$ we have that 
$\{h_n(g_1)\cdot C^k_{\F}(c) \ | \ n<\omega\}$ is finite. Thus, by Lemma \ref{LeftCosets}, the number of cosets of $C^k_{\F}(c)$ 
by the image of $g_1$ under any morphism from $G$ to $\F$ that stays the identity on $\F$, is bounded by $k$.  
\end{proof}

Similar results hold for double cosets.

\begin{lemma}\label{DoubleCosets}
Let $G$ be a group that has the structure of a tower $\mathcal{T}(G,\F)$. Let 
$(h_n)_{n<\omega}:G\to\F$ be a test sequence for $\mathcal{T}(G,\F)$ and $\gamma,\delta\in\F\setminus\{1\}$.  

Let $p,q$ be natural numbers and 
$g\in G$ such that $\{C_{\F}^p(c)\cdot h_n(g)\cdot C^q_{\F}(d) \ | \ n<\omega\}$ is finite. Then one of the following 
holds: 
\begin{itemize}
 \item either $g$ belongs to $\F$; or 
 \item the elements $c$ and $d$ are not conjugates and there exist abelian flats $G^{i+1}:= G^i*_A(A\oplus\Z^m)$ and 
 $G^{j+1}:= G^j*_B(B\oplus\Z^r)$ for $j<i$ and an element $e$, 
 with $A=C_{\F}(c)$, $B=C_{\F}(d)$ and such that 
 $g=z\cdot \gamma$, where $z$ belongs to $(A\oplus\Z^m)$ and $e^{-1}\cdot\gamma$ belongs to $B\oplus\Z^r$ ; or 
 \item the element $d$ is a conjugate of $c$ by the element $\gamma^{-1}e$ of $\F$ and there exist an abelian flat $G^{i+1}:= G^i*_A(A\oplus\Z^m)$, 
 where $A=C_{\F}(c)$. Moreover,  $g$ has the normal form $z_1\gamma z_2\gamma^{-1}e$ with respect to $G^i*_A(A\oplus\Z^m)$.
\end{itemize} 
In particular, there are at most $p\cdot q$ double cosets $C_{\F}^p(c)\cdot h_n(g)\cdot C^q_{\F}(d)$ under any morphism $h$ from $G$ to $F$ 
that stays the identity on $\F$.
\end{lemma}
\begin{proof}
The proof is by induction on the height of $g$. 

Since $\{C_{\F}^p(c)\cdot h_n(g)\cdot C^q_{\F}(d) \ | \ n<\omega\}$ is finite, we may assume that $h_n(g)=c^{k_n}ed^{s_n}$, where $c$ and $d$ have no roots 
and $k_n, s_n$ tend to infinity as $n$ goes to infinity. 

For the base case, suppose the height of $g$ is $0$ and $G^0:=\F*\F_l$. We can use the Bestvina-Paulin method 
to obtain a limiting action on a based real tree $(Y,*)$ satisfying the conclusions of Fact \ref{GroundFloorFact}. 
Under the above assumptions, either $g$ fixes the base point and thus $g$ is in $\F$, 
or the segment $[*,g\cdot *]$ is not trivial. In the latter case, either a nontrivial initial segment $I_1$ of $[*,g\cdot *]$ is fixed by $C_{\F}(c)$ 
or an initial segment $I_2$ of $[*,g^{-1}\cdot *]$ is fixed by $C_{\F}(d)$. In both cases we have a contradiction, since no segment in this 
real tree can have a nontrivial stabiliser. 

Assume the result is true for any element $g$ of height less than or equal to $i$. We show it for elements of height $i+1$. 
Recall that by the definition of a tower the $i+1$-th flat could either be a surface flat or a free abelian flat. In both cases, 
we follow the argument of the base case. We may assume that $g$ does not fix the base point, hence either the nontrivial 
segment $[*,g\cdot *]$ is fixed by the centraliser $C_{\F}(c)$ or the nontrivial segment $[*,g\cdot *]$ is fixed by the centraliser 
$C_{\F}(d)$.  

Suppose first that the $i+1$-th flat is a surface flat. As in the base case, all nontrivial segments are trivially stabilised, thus we obtain 
a contradiction. Therefore, the $i+1$-th flat is a free abelian flat. We further separate cases as follows: \\ \\
{\em Case 1:} We first consider the case where either $(c^{k_n})$ or $(d^{s_n})$ 
approximate the base point. Without loss of generality assume $(d^{s_n})$ approximates it. We show that the segment $[*,g\cdot *]$ 
must be entirely contained in a translate of the component, say $Y_u$, which is acted on by the abelian flat. 
Indeed, the segment $[*,g\cdot *]$ is stabilised by $C_{\F}(c)$, since for any element $c^k$ from $C_{\F}(c)$, the distance 
$\hat{d}_n(c^{k_n}ed^{s_n}, c^kc^{k_n}ed^{s_n})$ tends to zero. Thus, as in Lemma \ref{UniCent}, we may assume that the peg of the 
free abelian flat is the centraliser $C_{\F}(c)$ and moreover $g\cdot *=z\cdot *$ for some $z\in A\oplus \Z^m$. 
Thus, $g= z\cdot\gamma$ for some $\gamma\in G^i$. Furthermore, 
without loss of generality, we may assume that $h_n(z)=c^{k_n}$, and therefore $h_n(\gamma)=ed^{s_n}$. 
By Lemma \ref{UniCent} we get that there exists a free abelian flat $G^{j+1}:=G^j\oplus(B\oplus\Z^r)$ with $B=C_{\F}(d)$ 
and $e^{-1}\gamma$ belongs to $B\oplus\Z^r$. In addition, by the construction of the tower $c$ and $d$ cannot be conjugates. 
\ \\ \\
{\em Case 2:} We are left with the case where neither $(c^{k_n})$ nor $(d^{s_n})$ approximate the base point. In this 
case the segment $I_0:=[*, (c^{k_n})]$ is an initial segment of $[*, g\cdot *]$. If not, all $d$, $c$, and $e$ 
must commute and this case have been dealt in Lemma \ref{UniCent}, where the conclusion is as desired. 

The segment $I_0$ is fixed by the centraliser $C_{\F}(c)$, thus, we may assume as before that $A=C_{\F}(c)$. We next 
observe that no nontrivial subsegment of $I_1:=[(c^{k_n}), c^{k_n}ed^{s_n}]$ can be fixed by a power of $c$, thus $(c^{k_n})$ 
is a branching point. In particular there exists $z_1$ in $A\oplus \Z^m$ so that $I_0=[*, z_1\cdot *]$. Without loss 
of generality we may assume that $h_n(z_1)=c^{k_n}$. We next see that $I_1$ is fixed by 
$z_1eC_{\F}(d)e^{-1}z_1^{-1}$. This follows by verifying that $\hat{d}_n(c^{k_n}ed^{s_n}, c^{k_n}ed^le^{-1}c^{-k_n}c^{k_n}ed^{s_n})$ tends to zero. 
Thus, by the structure of the limiting tree, we have that there exists $\gamma$ in $G^i$ such that 
$z_1eC_{\F}(d)e^{-1}z_1^{-1}= z_1\gamma C_{\F}(c)\gamma^{-1}z_1^{-1}$. In particular, $\gamma$ must belong to $\F$ and $d$ is a 
conjugate of $c$ by $\gamma^{-1}e$. Since $[*, g\cdot *]$ is contained in the union of $Y_u$ (the component which is acted on by $A\oplus \Z^m$) and 
its translate $z_1\cdot\gamma Y_u$, we see that the normal form of $g$, with respect to the free abelian flat, must be $z_1\gamma z_2\delta$ for 
some $z_2$ in $A\oplus \Z^m$ and some $\delta$ in $G^i$. Therefore, its images must satisfy $h_n(g)=c^{k_n}\gamma c^{l_n}h_n(\delta)=
c^{k_n}ed^{s_n}=c^{k_n}ee^{-1}\gamma c^{s_n}\gamma^{-1}e$, 
hence $h_n(\delta)=c^{s_n-l_n}\gamma^{-1}e$. Now observe that, by the structure of a tower, we cannot have a conjugate of $C_{\F}(c)$ being 
a peg in a floor at a lower level. In particular, the element $\delta$ must be an element of $\F$ of the form $c^l\gamma^{-1}e$, as desired. 
\end{proof}

\begin{theorem}\label{BoundDoubleCosets}
Let $G$ be a group that has the structure of a tower $\mathcal{T}(G,\F)$. Let 
$(h_n)_{n<\omega}:G\to\F$ be a test sequence for $\mathcal{T}(G,\F)$. Let $g_1,g_2,g_3$ be elements in $G$. 
Let $k,l$ be natural numbers and $E_{4_{k,l}}$ be the equivalence relation of $k,l$-double-cosets 
(see Definition \ref{Imaginaries}). Suppose $\{[(h_n(g_1),h_n(g_2),h_n(g_3))]_{E_{4_{k,l}}} \ | \ n<\omega\}$ is finite. 
Then the set $\{[(h(g_1),h(g_2),h_n(g_3))]_{E_{4_{k,l}}} \ | \ h:G\rightarrow\F, \ \ h\upharpoonright\F=Id \}$ is bounded by $k\cdot l$. 
\end{theorem}
\begin{proof}
Since the set $\{[(h_n(g_1),h_n(g_2),h_n(g_3))]_{E_{4_{k,l}}} \ | \ n<\omega\}$ is finite we have that the sets  
$\{C_{\F}(h_n(g_1)) \ | \ $ $ n<\omega\}$ and $\{C_{\F}(h_n(g_3)) \ | \ $ $ n<\omega\}$  are finite. Thus, by Lemma \ref{UniCent} 
there exist non-trivial elements $c,d\in\F$ so that, 
for any morphism from $G$ to $\F$ that stays the identity on $\F$, the centraliser of the image of $g_1$ is $C_{\F}(c)$ and the 
centraliser of the image of $g_3$ is $C_{\F}(d)$.  

On the other hand, again by the finiteness of $\{[(h_n(g_1),h_n(g_2),h_n(g_3))]_{E_{4_{k,l}}} \ | \ n<\omega\}$ we have that 
$\{C^k_{\F}(c) \cdot h_n(g_2)\cdot C^k_{\F}(d) \ | \ n<\omega\}$ is finite. Thus, by Lemma \ref{DoubleCosets}, 
the number of double cosets $C^k_{\F}(c) \cdot h(g_2)\cdot C^k_{\F}(d)$ 
by the image of $g_2$ under any morphism from $G$ to $\F$ that stays the identity on $\F$, is bounded by $k\cdot l$.  
\end{proof}

\section{Diophantine envelopes}\label{EnvelopeGraded}
In this section we present {\em Diophantine envelopes}. A tool introduced by Sela in \cite{SelaImaginaries}. 

Before recording the main result of this section we need to explain some construction. To each solid limit group $Sld$ 
with a nontrivial $JSJ$ decomposition we may assign a group $Comp(Sld)$, 
its {\em completion}, with the following properties: 

\begin{itemize}
 \item the group $Comp(Sld)$ has the structure of a tower over the group $Sld$;
 \item there exists an embedding $f_t$ of $Sld$ into $Comp(Sld)$ whose image is not the solid limit group 
 in the bottom of the tower;
 \item suppose that $f_b$ embeds $Sld$ onto the solid limit group in the bottom of $Comp(Sld)$, then 
 for any two strictly solid morphisms, say $s_1, s_2:Sld\rightarrow \F$, that belong to the same strictly solid family, 
 there exists a morphism $S:Comp(Sld)\rightarrow \F$ such that $s_1=S\circ f_b$ and $s_2=S\circ f_t$. 
\end{itemize}

Using the completion of a solid limit group $Sld$ we may transform a tower $\mathcal{T}(G,Sld)$ over it in the following way: 

\begin{itemize}
 \item since $f_t$ embeds $Sld$ into $Comp(Sld)$ we replace the bottom group 
 of the tower with the group $Comp(Sld)$ by gluing the edges at the image of $Sld$ 
 by $f_t$; 
 \item since $Comp(Sld)$ has the structure of a tower over $Sld$, we expand its floors; 
 \item we merge pegs of free abelian floors that might happen to be conjugates of pegs that existed in the tower 
 $\mathcal{T}(G,Sld)$. 
\end{itemize}

The above construction takes as an input a tower over a solid limit group $\mathcal{T}(G,Sld)$ and outputs 
$\mathcal{GT}(G*_{Sld}Comp(Sld),Sld)$ a tower over the same solid limit group but with some intermediate floors 
coming from the tower structure of the completion of the solid limit group. We call 
such a tower a {\em graded tower}. 

Consider a graded tower $\mathcal{GT}(G*_{Sld}Comp(Sld),Sld)$ and a strictly solid family of morphisms 
$[s]:Sld\to\F$. To this pair of objects we assign a tower $\mathcal{GT}_{s}$ over $\F$ by replacing the bottom solid limit group 
with $\F$ and glue the edges on their images by some strictly solid morphism $s$ in the strictly solid family 
$[s]$. The resulting tower might depend on the choice of $s$, but we can still recover any other strictly solid morphism in 
the family of $s$ by the properties of the completion $Comp(Sld)$. The group associated to $\mathcal{GT}_{s}$ 
is the group generated by the groups $G*_{Sld}Comp(Sld)$ and $\F$ modulo the relation imposed by the morphism 
$s:f_b(Sld)\rightarrow \F$. In particular, to any element of $G*_{Sld}Comp(Sld)$ we can assign canonically 
its homomorphic image to the group corresponding to $\mathcal{GT}_{s}$. 

We define a {\em graded test sequence} of 
$\mathcal{GT}(G*_{Sld}Comp(Sld),Sld)$ with respect to the strictly solid family $[s]$ to be a test sequence of some 
(ungraded) tower associated to $\mathcal{GT}(G*_{Sld}Comp(Sld),Sld)$ and $[s]$ in the way explained above.
 
Finally, when we want to be explicit about a generating set, say $\bar{u}$, of a group $G$ that has the structure of a tower 
$\mathcal{T}(G,H)$, then we denote the tower by $\mathcal{T}(G,H)[\bar{u}]$. 
 
\begin{fact}[Diophantine envelope]\label{GradedEnvelope}
Suppose $\phi(\bar{x},\bar{y})$ is a first order formula over a nonabelian free group $\F$. 
Then there exist finitely many graded towers $\mathcal{GT}^1(G_1, Sld^1)[(\bar{u},$ $\bar{x},\bar{y})], 
\ldots,\mathcal{GT}^k(G_k,$ $Sld^k)[(\bar{u},\bar{x},\bar{y})]$, such that:
\begin{itemize}
 \item[(i)] for each $i\leq k$, the solid limit group $Sld^i$ is solid with respect to 
 the subgroup $\langle\bar{y}\rangle$; 
 \item[(ii)] for each $i\leq k$, there exists  
 a strictly solid family of morphisms $[s]:Sld^i\to\F$ and 
 a graded test sequence, say $(h_n)_{n<\omega}$, that corresponds to $[s]$, so that 
 $\F\models \phi(h_n(\bar{x}),h_n(\bar{y}))$ for every $n<\omega$;
 \item[(iii)] If $\F\models\phi(\bar{b},\bar{c})$, then there exist some $i\leq k$ and:
 \begin{itemize}
 \item[\textbullet] a morphism $h:G^i\to\F$, such that $h(\bar{x},\bar{y})=(\bar{b},\bar{c})$.  
 \item[\textbullet] a strictly solid family of morphisms $[s]:Sld^i\to\F$ with $s(\bar{y})=\bar{c}$ and a graded test sequence 
 $(h_n)_{n<\omega}$ of $\mathcal{GT}^i$ that corresponds to $[s]$ such that $\F\models\phi(h_n(\bar{x}),\bar{c})$. 
 \end{itemize}
\end{itemize}
\end{fact}

\begin{remark}
\ 
\begin{itemize}
\item In the previous fact, by an abuse of language, we refer to $\bar{x}$ (respectively $\bar{y}$) and the tuple it is 
assigned to in the ungraded tower by the same name;
\item If we want to refer to the third part of this theorem with respect to some solution $(\bar{b},\bar{c})$ of $\phi$ in $\F$, 
we will say that $(\bar{b},\bar{c})$ {\em factors through} the graded tower $\mathcal{GT}^i$ and the strictly solid family $[s]$. 
\end{itemize}
\end{remark}

\section{The finite cover property}\label{FCP}
In this section we prove the main result of the paper. We start by proving that the ``there exists infinitely many'' quantifier 
can be eliminated in the real sort.

\begin{theorem}\label{RealSort}
Let $\phi(x,\bar{y})$ be a first order formula over a non abelian free group $\F$. Then there exists $n<\omega$ 
such that, for any $\bar{c}\in\F$, if $\phi(x,\bar{c})$ is finite, then $\abs{\phi(x,\bar{c})}\leq n$.
\end{theorem}
\begin{proof}
Suppose not, and $(\bar{c}_l)_{l<\omega}$ be a sequence of tuples such that, for each $l$, $l<\abs{\phi(\bar{x},\bar{c}_l)}<\infty$.  

We consider the Diophantine envelope, $\mathcal{GT}^1(G_1,Sld^1)[(\bar{u}, x,\bar{y})]
,\ldots,\mathcal{GT}^k(G_k,$ $Sld^k)[\bar{u},x,\bar{y}]$, of $\phi(x,\bar{y})$  as given by Fact \ref{GradedEnvelope}. 
Let $r_i$ be the bound on the number of strictly solid families of morphisms for the solid limit group $Sld^i$ 
in the bottom floor of $\mathcal{GT}^i$ as given by Theorem \ref{BoundSolid}.

Let $n$ be a natural number with $n>r_1+\ldots+r_k$. By Fact \ref{GradedEnvelope}(iii), each solution of $\phi(x,\bar{c}_n)$ 
factors through some graded tower in the Diophantine envelope of $\phi$. Fix a solution $b$ and consider 
the graded tower $\mathcal{GT}^i$ and the strictly solid family of morphisms $[s]$ that $(b,\bar{c_n})$ factor through. 
Since $\phi(x,\bar{c}_n)$ is finite we must have 
that $x$ takes finitely many values under the test sequence of the ungraded tower $\mathcal{GT}[s]^i$. 
Thus, by Theorem \ref{UniRealSort}, $x$ takes a unique value for each morphism of the (ungraded) tower $\mathcal{GT}[s]^i$. 
In addition, by the first bullet of Fact \ref{GradedEnvelope}(iii), it must take the value $b$. In particular, any solution of $\phi(x,\bar{c}_n)$ 
that factors through the graded tower $\mathcal{GT}^i$ and the strictly solid family of morphisms $[s]$ must take the value $b$. 
Hence, we can have at most $r_i$ solutions that factor through it, a contradiction.
\end{proof}

We continue with the sort for conjugacy classes.

\begin{theorem}\label{ConjugacySort}  
Let $\phi(x,\bar{y})$ be a first order formula over a non abelian free group $\F$. Then there exists $n<\omega$ 
such that, for any $\bar{c}\in\F$, if $\phi(x,\bar{c})$ has finitely many solutions up to conjugation, 
then the number of conjugacy classes $\abs{\phi(x,\bar{c})}_{E_1}$ is bounded by $n$.
\end{theorem}
\begin{proof}
The proof is identical to the proof of Theorem \ref{RealSort}, using Theorem \ref{UniConjSort} instead of Theorem \ref{UniRealSort}.
\end{proof}

\begin{theorem}\label{CosetSort}  
Let $\phi(x_1,x_2,\bar{y})$ be a first order formula over a non abelian free group $\F$. Then there exists $n<\omega$ 
such that, for any $\bar{c}\in\F$, if $\phi(\bar{x},\bar{c})$ has finitely many solutions up to the equivalence relation $E_{2,m}$, 
then the number of $m$-left-cosets $\abs{\phi(\bar{x},\bar{c})}_{E_{2,m}}$ is bounded by $n$.
\end{theorem}
\begin{proof}
The proof is identical to the proof of Theorem \ref{RealSort}, by choosing $n>m\cdot (r_1+\ldots+r_k)$ and using Theorem \ref{BoundLeftCosets} 
instead of Theorem \ref{UniRealSort}.
\end{proof}

\begin{theorem}\label{DoubleCosetSort}  
Let $\phi(x_1,x_2,x_3,\bar{y})$ be a first order formula over a non abelian free group $\F$. Then there exists $n<\omega$ 
such that, for any $\bar{c}\in\F$, if $\phi(\bar{x},\bar{c})$ has finitely many solutions up to the equivalence relation $E_{4_{p,q}}$, 
then the number of $p,q$-double-cosets $\abs{\phi(\bar{x},\bar{c})}_{E_{4_{p,q}}}$ is bounded by $n$.
\end{theorem}
\begin{proof}
The proof is identical to the proof of Theorem \ref{RealSort}, by choosing $n>p\cdot q\cdot (r_1+\ldots+r_k)$ and using 
Theorem \ref{BoundDoubleCosets} instead of Theorem \ref{UniRealSort}.
\end{proof}

Finally combining Theorems \ref{RealSort}, \ref{ConjugacySort}, \ref{CosetSort} and \ref{DoubleCosetSort}, we get:

\begin{theorem}
Let $\F$ be a non abelian free group. Then $\mathcal{T}h(\F)^{eq}$ eliminates the $\exists^{\infty}$-quantifier.  
\end{theorem}
\begin{proof} 
Suppose, for the sake of contradiction, that there is a first order formula $\phi(x_E,\bar{y})$, where $x_E$ is 
a variable of the sort $S_E$ for some $\emptyset$-definable equivalence relation $E$ and $\bar{y}$ is a tuple of variables 
in the real sort, and a sequence of tuples $(\bar{c}_n)_{n<\omega}$ in $\Gamma$ such that $n<\abs{\phi(x_E,\bar{y})}<\infty$ 
for each $n<\omega$.

We consider the formula: 
$$\psi(x_1,\ldots,x_k,x_{E_1},\ldots,x_{E_l},\bar{y},\bar{a}):=\exists x_E(\phi(x_E,\bar{y})\land R_E(x_E,x_1,\ldots,x_k,x_{E_1},\ldots,x_{E_l},\bar{a}))$$ 

Where $R_E$ is the definable relation assigned to the equivalence relation $E$ by Theorem \ref{Elim}, $x_1,\ldots,x_k$ are variables 
from the real sort and $x_{E_j}$ are variables of the sort $S_{E_j}$ for some basic equivalence relations $E_j$ (see Definition \ref{Imaginaries}). 
Then, by Theorem \ref{Elim}, the sequence of tuples $(\bar{c}_n,\bar{a})_{n<\omega}$ in $\F$ witnesses that 
$n<\abs{\psi(x_1,\ldots,x_k,x_{E_1},\ldots,x_{E_l},\bar{c}_n,\bar{a})}<\infty$ for each $n<\omega$. 

Finally, by Lemma \ref{ImagSorts}, there exists a first order formula with all variables in the real sort  
$\theta(x_1,\ldots,x_k,\bar{z}_1,\ldots,\bar{z}_l,\bar{y},\bar{w})$ such that 
$$\F^{eq}\models \forall x_1,\ldots,x_k,\bar{z}_1,\ldots,\bar{z}_l,\bar{y},\bar{w}
(\psi(x_1,\ldots,x_k,f_{E_1}(\bar{z}_1),\ldots f_{E_l}(\bar{z}_l),\bar{y},\bar{w})\leftrightarrow$$
$$\theta(x_1,\ldots,x_k,\bar{z}_1,\ldots,\bar{z}_l,\bar{y},\bar{w}))$$. 

Thus, after refining, re-enumerating and possibly expanding the sequence (that we still denote by) $(\bar{c}_n,\bar{a})$, one of the following holds:
\begin{itemize}
 \item there exists a first order formula $\theta_1(z,\bar{y},\bar{w})$ such that 
 $n<\abs{\theta_1(z,\bar{c}_n,\bar{a})}<\infty$; 
 \item for some natural number $m$, there exists a first order formula $\theta_{2_m}(z_1,z_2,\bar{y},\bar{w})$ 
 such that $n<\abs{\theta_{2_m}(z_1,z_2,\bar{c}_n,\bar{a})}_{E_{2_m}}<\infty$; 
 \item for some couple of natural numbers $(m,k)$, there exists a first order formula $\theta_{4_{m,k}}(z_1,z_2,$ $z_3,\bar{y},\bar{w})$ 
 such that $n<\abs{\theta_{4_{m,k}}(z_1,z_2,z_3,\bar{c}_n,\bar{a})}_{E_{4_{m,k}}}<\infty$;
\end{itemize}
In each case a contradiction is reached.
\end{proof}

Together with the stability of the theory of non abelian free groups we get: 

\begin{corollary}
Let $\F$ be a non abelian free group. Then $\mathcal{T}h(\F)$ does not have the finite cover property.
\end{corollary}

\bibliography{biblio}
\end{document}